\documentclass[preprint, 11pt, nonatbib]{elsarticle}
\usepackage[utf8]{inputenc}
\usepackage{todonotes}
\usepackage{amssymb}
\usepackage{amsmath}
\usepackage{amsthm}
\usepackage{tikz-cd}
\usepackage{verbatim}

\makeatletter
\let\c@author\relax
\makeatother

\usepackage[
backend=biber,
maxnames=99,
sorting=nty,doi=false,url=false
]{biblatex}
\DeclareMathOperator{\Val}{Val}

\DeclareMathOperator{\ord}{ord}
\DeclareMathOperator{\Spec}{Spec}
\DeclareMathOperator{\Zar}{Zar}

\DeclareMathOperator{\InitialForm}{in}
\DeclareMathOperator{\AssociatedGraded}{gr}

\newtheorem{theorem}{Theorem}

\newtheorem{corollary}[theorem]{Corollary}
\newtheorem{lemma}[theorem]{Lemma}

\theoremstyle{definition}

\newtheorem{example}[theorem]{Example}
\newtheorem{question}[theorem]{Question}

\numberwithin{theorem}{section}

\begin{document}

\begin{frontmatter}

\title{Connectedness and integrally closed local overrings of two-dimensional regular local rings}

\author[Heinzer]{William Heinzer}
\address[Heinzer]{Department of Mathematics, Purdue University, West
Lafayette, Indiana 47907-1395 U.S.A.}
\ead{heinzer@purdue.edu}

\author[Loper]{K.~Alan Loper}
\address[Loper]{Department of Mathematics, Ohio State University -- Newark,  Newark, OH 43055}
\ead{lopera@math.ohio-state.edu}

\author[Olberding]{Bruce Olberding\fnref{grant} } 
\address[Olberding]{Department of Mathematical Sciences, New Mexico State University,
 Las Cruces, NM 88003-8001 U.S.A.}
 \ead{olberdin@nmsu.edu}

\author[Toeniskoetter]{Matthew Toeniskoetter} 
\address[Toeniskoetter]{Department of Mathematics and Statistics, Oakland University, Rochester, MI 48309-4479}
 \ead{toeniskoetter@oakland.edu}

\begin{abstract}
Let $D$ be a two-dimensional regular local ring. 
We prove there is a one-to-one correspondence between closed connected sets in the space of valuation overrings of $D$ that dominate $D$ and the integrally closed local overrings of $D$ that are not essential valuation rings   or divisorial valuation rings of $D$. 

\end{abstract}

\fntext[grant]{The author was supported by NSF grant DMS-2231414.}

\begin{keyword}
Regular local ring \sep
local ring \sep
valuation ring \sep
Zariski-Riemann space \sep
quadratic transform   

\MSC 13A18 \sep 13H05

\end{keyword}

\end{frontmatter}

\section{Introduction}

In \cite{HLOT}, we proved the following theorem, which can be viewed as a lifting of Zariski's Connectedness Theorem to spaces of valuation rings. In the theorem, we do not need to assume anything about $D$ other than that it is a local integral domain, nor do we need to assume $F$ is the quotient field of $D$. We denote by  $\Val_F(D)$  the set of valuation rings of the field $F$ that dominate  the local ring $D$ (i.e., the maximal ideal   of $V$ contains the maximal ideal   of $D$), and whose topology is the Zariski topology with basic open sets given by $$\{V \in \Val_F(D):x_1,\ldots,x_n \in V\}, {\mbox{ where }} x_1,\ldots,x_n \in F.$$

\begin{theorem}{\em \cite[Theorem 1.1]{HLOT}}
\label{VCT} \label{is local}
Let $F$ be a field and let $D$ be a local subring of $F$.
A subring $A$ of $F$ dominating $D$ is local, residually algebraic over $F$, and integrally closed in $F$ if and only if there is a closed and connected subspace $Z$ of $\Val_F(D)$ such that $A$ is the intersection of the rings in $Z$. In this case, $Z$ may be taken to be $\Val_F(D)$. 
\end{theorem} 

Thus  the integrally closed local rings between $D$ and $F$ that dominate $D$ and are residually algebraic over $D$ are precisely the intersections of rings from closed connected subsets of $\Val_F(D)$. In this article, we pursue these ideas in an important context for birational algebra and geometry in which $D$ is a two-dimensional regular local ring and $F$ is its quotient field.
 In this case we prove a tighter version of Theorem~\ref{is local}, which can be viewed as a kind of Nullstellensatz for integrally closed local overrings of $D$. Where this theorem differs from Theorem~\ref{is local} is that it establishes a one-to-one correspondence between closed connected sets of $\Val(D):=\Val_F(D)$ and integrally closed local overrings that are not essential valuation rings or divisorial valuation rings of $D$. Thus part of the added strength in the following theorem is that it shows 
 that except in a couple of trivial cases, 
 every integrally closed local  overring $A$ of $D$
is not only  represented as an intersection of valuation rings from a closed and connected set in $\Val(D)$ but that there is only one choice of such a representing set, namely 
the collection of valuation rings in $\Val(D)$ that dominate $A$. 
 For a nonempty subset $Z$ of $\Val(D)$, we denote by $A(Z)$ the intersection of the rings in $Z$. 

\begin{theorem} \label{nullstellensatz} With $D$ a two-dimensional regular local ring, the mappings 
\begin{center} $Z \mapsto A(Z)$ and $A \mapsto \Val(A)$ 
\end{center} 
define  a one-to-one correspondence between closed connected subsets of $\Val(D)$ and  integrally closed local overrings of $D$ that are not essential or divisorial valuation rings of $D$. 
\end{theorem}

The proof of this theorem relies on irredundance properties of intersections of local rings in nonsingular projective models over $D$. We develop these properties in Section 3 and prove, among other things, that the intersection of any set of incomparable two-dimensional regular local overrings of $D$ is irredundant.  

We prove Theorem~\ref{nullstellensatz} in Section~4, where we also
 use this correspondence to describe some strong properties of the space of valuation overrings of integrally closed local overrings of $D$.  For example, it is proved that if 
  ${Z}$ is a  subset of ${\rm Val}(D)$ such that $A({Z})$ is local, then ${Z}$ is   connected (Theorem~\ref{two dim}) and that if
 $A$ is an integrally closed local overring of $D$ that dominates $D$, the only patch closed representation of $A$ in $\Val(D)$ is $\Val(A)$ (Theorem~\ref{pre vacant}). As a consequence, we obtain that  every one-dimensional integrally closed local overring $A$ of $D$ is vacant (Corollary~\ref{vacant cor}), meaning that up to patch closure there is only one way to write $A$ as an intersection of valuation overrings.
 
\section{Preliminaries} 

{\it For the rest of the article, let $D$ denote a two-dimensional regular local ring with quotient field $F$.} 
In this section we collect some definitions and observations that will be needed in the next sections. 

\medskip

 \noindent {\bf (2.1)} 
 A valuation overring $V$ of $D$ that is not the quotient field  of $D$ has Krull dimension $1$ or $2$. If  $V$ has Krull dimension $2$, then $V$ is discrete, while if $V$ has Krull dimension $1$, then the value group of $V$ has rational rank $1$ or  $2$ \cite[Theorem~1]{MR0082477}.  If $V$ is a DVR whose residue field is transcendental over the residue field of $D$, then $V$ is a {\it divisorial valuation ring} of $D$. If  $V$ is  a localization of $D$ at a height one prime ideal of $D$, then $V$ is an {\it essential valuation ring} of $D$.
 The integrally closed local overrings of $D$ that dominate $D$ and are residually transcendental over $D$ are precisely the divisorial valuation rings of $D$ \cite[Theorem~1]{MR0082477}.

\medskip

 \noindent {\bf(2.2)} 
 The {\it Zariski-Riemann space} $\Zar(D)$ of  $D$ is the set of valuation overrings of $D$ with
the Zariski topology, which is defined as for $\Val(D)$ in the introduction but with quantification over $\Zar(D)$ rather than $\Val(D)$, so that  $\Val(D)$ is a subspace of $\Zar(D)$.   
 The Zariski topology on $\Zar(D)$ (and hence also on its subspace $\Val(D)$) can be refined to a zero-dimensional Hausdorff topology called the {\it patch topology,} which has as a basis of clopen sets the sets of the form $$\{V \in \Zar(D):x_1,\ldots,x_n \in V, y \not \in V\}, {\mbox{ where }} x_1,\ldots,x_n,y\in F.$$
 If $Z$ is a patch closed subset of $\Zar(D)$, then it follows as in 
 \cite[Tag 0903]{stacks-project} that
 the Zariski closure $\overline{Z}$ of $Z$ in $\Zar(D)$  is $$\overline{Z} = \{V \in \Zar(D):V \subseteq U {\mbox{ for some }} U \in Z\}.$$

 \noindent {\bf (2.3)} If $V$ is a divisorial valuation ring of $D$ as in (2.1), then there are infinitely many minimal valuation overrings of $D$ contained in $V$.
Furthermore, $V$ is a patch limit point of every infinite subset of these; see for example \cite[Corollary~3.4]{MR3299704}.

\medskip

 \noindent {\bf(2.4)} A {\it local quadratic transform} of $D$ is an overring  of $D$ of the form $D[x/y]_P$, where $x,y$ generate the maximal ideal ${\mathfrak{m}}$ of $D$ and $P$ is a prime ideal of $D[x/y]$ that contains  ${\mathfrak{m}}$. A local quadratic transform of~$D$ is  a regular local  
ring of Krull dimension $1$ or $2$. If $D \subseteq D_1 \subseteq \cdots \subseteq D_n$ are local rings with each $D_{i+1}$ a local quadratic transform of $D_{i}$, we say $D_{i+1}$ is an  {\it iterated local quadratic transform} of $D$.
 Every two-dimensional regular local overring of $D$ is an iterated local quadratic transform of $D$ \cite[Theorem 3]{MR0082477}, and so the set of two-dimensional regular local overrings of $D$ is a tree with respect to set inclusion.
 We denote this tree, the {\it quadratic tree of $D$}, by $Q(D)$.
 The tree $Q(D)$ can be stratified as a union of sets $Q_i(D)$, where for each $i$, $Q_i(D)$ is the set of regular local rings in $Q(D)$ that are obtained by $i$ iterated local quadratic transforms.
 We denote by $Q_{\leq n}(D)$ the union of the sets $Q_i(D)$, $i \leq n$. 
 For more on the tree $Q(D)$, see \cite{MR0340250,MR4108340}.
 A valuation overring of $D$ dominates $D$ if and only if it is a union of a chain in $Q(D)$ \cite[Lemma~12]{MR0082477}.  

\section{Irredundant intersections of regular local overrings} 

In this section we consider intersections of two-dimensional regular local overrings of $D$, with special emphasis on the irredundance of intersection representations. If $X$ is a nonempty subset of the quadratic tree $Q(D)$ discussed in (2.4), the intersection $A$ of the rings in $X$ is {\it irredundant} if no ring in $X$ can be omitted from the representation of $A$ as an intersection of rings in $X$. The property of irredundance plays an important role in the main proofs of the next section. 

Our analysis of irredundance depends on the behavior of these intersections with respect to localizations of $D$ at height-one prime ideals. 
The next lemma, which asserts that to each ring in a set of incomparable rings in $Q(D)$ can be associated a unique localization of $D$ at a height-one prime ideal, is proved in \cite[Lemma~8.5]{MR4097907} under an assumption on $D$ that holds when $D$ is Henselian. We are able to remove this assumption in the current proof.

\begin{theorem}\label{unique essential}
Let $X$ be a collection of two-dimensional regular local overrings of $D$ that are pairwise incomparable with respect to inclusion. Then for each $\alpha \in X$, there is a height $1$ prime ideal ${\mathfrak p}$ such that $D_{\mathfrak p}$ contains $\alpha$ but no other ring in $X$. 
\end{theorem}

To prove this theorem, we use two lemmas from 
\cite[Appendix 5]{MR0389876}. To state these, we recall several definitions. 
For a Noetherian local ring $R$ and a nonzero element $f \in R$, the \emph{order} of $f$ with respect to the maximal ideal $\mathfrak{m}$ of $R$, denoted $\ord (f)$, is the unique integer $r$ such that $f \in \mathfrak{m}^r \setminus \mathfrak{m}^{r+1}$.
If $R$ is a regular local ring, then its order function $\ord$ extends to a valuation on its quotient field and its associated valuation ring is the \textit{order valuation ring} for $R$.
With $\ord (f) = r$, the \textit{initial form} of $f$, denoted $\InitialForm (f)$, is the image of $f$ in the degree $r$ component $\mathfrak{m}^r / \mathfrak{m}^{r+1}$ of the associated graded ring $\AssociatedGraded_{\mathfrak{m}} (R) = R [\mathfrak{m} t] / \mathfrak{m} R [\mathfrak{m} t]$, where $t$ is an indeterminate. 
For an ideal $I \subseteq R$, its initial form ideal is the ideal generated by its initial forms, $\InitialForm (I) = ( \InitialForm (f) \mid f \in I ) \AssociatedGraded_{\mathfrak{m}} (R)$.
We will use the fact that the prime factors of the initial form of an element of a two-dimensional regular local ring $\alpha$ determine the finitely many points in the first neighborhood of $\alpha$ contained in the essential valuation rings of $\alpha$.

\begin{lemma}{\em \cite[Appendix 5, Proposition 3]{MR0389876}}\label{approximation}
Let $\alpha$ be a two-dimensional regular local ring, let $f \in \mathfrak{m}_\alpha$, and let $\InitialForm (f) = G H$ be a factorization in $\AssociatedGraded_\alpha (\mathfrak{m}_\alpha)$, where $G$ and $H$ are relatively prime.
Then for any $n > \ord (f)$, there exist $g, h \in \alpha$ such that $\InitialForm (g) = G$, $\InitialForm (h) = H$, and $f - g h \in {\mathfrak{m}_\alpha}^n$.
\end{lemma}

\begin{lemma}{\em \cite[Appendix 5, Proposition 2]{MR0389876}}\label{directions}
Let $\alpha$ be a two-dimensional regular local ring, let $f \in \mathfrak{m}_\alpha$ be a prime element with associated essential valuation ring $V = D_{(f)}$, and let $(\InitialForm (f)) = \mathfrak{p}_1^{a_1} \cdots \mathfrak{p}_s^{a_s}$, where $\mathfrak{p}_i$ are distinct prime ideals and $a_1, \ldots, a_s > 0$.
Then the $\mathfrak{p}_i$ are in one-to-one correspondence with the points $\alpha' \in Q_1 (D)$ such that $V \supseteq \alpha'$.
\end{lemma}

Using the above two lemmas and implicitly using the idea of the transform of an ideal as in \cite[Appendix 5]{MR0389876}, \cite{MR491722}, and \cite{MR0977761}, we prove Lemma~\ref{PrimeLemma}, which is a rephrasing of Theorem~\ref{unique essential}.
The technique of the proof is to start with an arbitrary essential valuation ring of a point $\beta$ in the quadratic tree $Q(D)$ from (2.4), then to iteratively approximate this essential valuation ring down the quadratic tree towards $D$ such that the approximations only occur along the path back towards $\beta$.

\begin{lemma}\label{PrimeLemma}
Let $\beta \in Q (D)$.
Then there exists an essential valuation ring of $D$ and $\beta$, say $V$, such that for all $\gamma \in Q (D)$, if $V$ is an essential valuation ring of $\gamma$, then $\gamma$ is comparable to $\beta$.
\end{lemma}

\begin{proof}
Let $D = \alpha_0 \subsetneq \alpha_1 \subsetneq \cdots \subsetneq \alpha_n = \beta$ be the unique sequence of local quadratic transforms from $D$ to $\beta$, and let $\mathfrak{m}_i$ denote the maximal ideal of $\alpha_i$.
Let $V$ be any essential valuation ring of both $D$ and $\beta$, i.e.\ let $V$ be any essential valuation ring of $\beta$ that isn't the order valuation ring of an $\alpha_i$.
We shall inductively modify $V$ until it satisfies the statement of the lemma.

For each $i$, denote $P_i = \mathfrak{m}_V \cap \alpha_i$, the center of $V$ on $\alpha_i$. As an essential valuation ring of $D$ that contains $\alpha_i$, $V$ is also an essential valuation ring of $\alpha_i$.
By Lemma~\ref{directions}, the valuation ring $V$ satisfies the statement of the lemma if and only if $\InitialForm (P_i)$ is a power of a homogeneous prime ideal for each $i$, $0 \le i < n$.

Suppose $V$ doesn't already satisfy the statement of the lemma, and let $k$ be the largest integer, $0 \le k < n$, such that $\InitialForm (P_k)$ is not the power of a homogeneous prime ideal.
To simplify notation, we assume $k = 0$, but the following argument applies for any value of $k$.
Choose a generator for $P_0$, say $P_0 = p_0 \alpha_0$.
For each $i$ with $0 \le i \le n$, fix an element $x_i \in \mathfrak{m}_i$ such that $\mathfrak{m}_i \alpha_{i+1} = x_i \alpha_{i+1}$, denote the order valuation of $\alpha_i$ by $\ord_i$,
and denote $r_i = \ord_i (P_i)$.
For each $i$ with $0 \le i \le n$, define
    $$p_i = \left( \prod_{0 \le j < i} {x_i}^{- r_j} \right) p_0,$$
so $p_i$ generates $P_i$.
For each $k$ from $0 \le k \le n + 1$, denote by $s_k$ th partial sum $s_k = \displaystyle \sum_{i=0}^{k-1} r_i$, and denote $m = \displaystyle \sum_{i=0}^{n} r_i = s_{n+1}$.

Write $\InitialForm (P_0) = \mathfrak{p}^{t} \mathfrak{a}$, where $\alpha_1$ corresponds to the homogeneous prime ideal $\mathfrak{p}$ and $\mathfrak{a}$ is relatively prime to $\mathfrak{p}$.
Using the fact that $\mathfrak{p}$ and $\mathfrak{a}$ are principal ideals, write $\InitialForm (p_0) = G H$ for some $G, H \in \AssociatedGraded_{\mathfrak{m}_0} (\alpha_0)$, where $(G) = \mathfrak{p}^t$ and $(H) = \mathfrak{a}$. 
Then find $g, h \in \alpha_0$ as in Lemma~\ref{approximation} such that $\InitialForm (g) = G$, $\InitialForm (h) = H$, and $p_0 - g h \in {\mathfrak{m}_0}^m$.
We claim that any height $1$ prime ideal containing $g$ satisfies the statement of the lemma.

Write $g_1 = g {x_0}^{- \ord_0 (g)}$ and $h' = h {x_0}^{- \ord_{0} (h)}$; notice that $\ord_0 (g) + \ord_0 (h) = r_0$ and $p_0 - g h = {x_0}^{r_0} (p_1 - g_1 h')$.
Since the ring $\alpha_1$ corresponds to the homogeneous prime ideal $\mathfrak{p}$ as in Lemma~\ref{directions}, and since $\mathfrak{a}$ is relatively prime to $\mathfrak{p}$, it follows that $h'$ is a unit in $\alpha_1$.
Since 
    $$p_0 - g h \in {\mathfrak{m}_0}^m \subseteq {x_0}^m \alpha_1,$$
it follows that 
    $$p_1 - g_1 h' = {x_0}^{- r_0} (p_0 - g h) \in {x_0}^{m - r_0} \alpha_1 \subseteq {\mathfrak{m}_1}^{m - r_0},$$
and since $m - r_0 > r_1$, we obtain $\InitialForm (p_1) = \InitialForm (g_1) \InitialForm (h')$.

Now, having defined $g_1$, inductively define $g_{i+1} = g_i {x_i}^{- r_i}$ for all $i$, $1 \le i <n$.
We prove by induction that $p_i - g_i h' \in {\mathfrak{m}}_i^{m - s_i}$ for all $i \leq n$, where the base case $i = 1$ has been verified above.
Suppose that this claim is true for $i$, where $1 \le i < n$.
Since $m - s_i > r_i$ and  $h'$ is a unit in $\alpha_i$, it follows that $\ord_i (g_i) = r_i$.
Then,
\begin{align*}
    p_{i+1} - g_{i+1} h' 
        &= \frac{p_i}{{x_i}^{r_i}} - \frac{g_i}{{x_i}^{r_i}} h' \\
        &= \frac{p_i - g_i h'}{{x_i}^{r_i}} \\
        &\in {x_i}^{-r_i}{\mathfrak{m}_i}^{m - s_i} \alpha_{i+1} \\
        &= {x_i}^{m - s_i - r_i} \alpha_{i+1} \\
        &= {x_i}^{m - s_{i+1}} \alpha_{i+1} \\
        &\subseteq {\mathfrak{m}_{i+1}}^{m - s_{i+1}},
\end{align*}
proving the claim.

Let $V'$ be an essential valuation ring of $\alpha_0$ such that $g \in \mathfrak{m}_{V'}$, and denote $P_i' = \mathfrak{m}_{V'} \cap \alpha_i$ for each $i$, $0 \le i \le n$, so $P_i'$ is a height $1$ prime containing $g_i \alpha_i$.
In $\alpha_0$, we have $(\InitialForm (g)) = \mathfrak{p}^t$, so $\InitialForm (P_i')$ is a power of $\mathfrak{p}$.
For $1 \le i < n$, since $\InitialForm (P_i') \supseteq (\InitialForm (g_i)) = \InitialForm (P_i)$ and $\InitialForm (P_i)$ is a power of a homogeneous prime ideal, so is $\InitialForm (P_i')$.
Thus $V'$ satisfies the statement of the lemma between $\alpha_k$ and $\beta = \alpha_n$, where $0 \le k < n$.
By repeating this argument, we construct $V'$ satisfying the statement of this lemma for $k = 0$, proving the lemma.
\end{proof}

The next lemma further clarifies the relationship between intersections of rings in the quadratic tree  $Q(D)$ from (2.4) and localizations of $D$ at height-one primes.

\begin{lemma} \label{localization}
Let $n\geq 1$,  let $ X $ be a nonempty subset of   $ Q_{\leq n}(D)$, and let ${\mathfrak{p}}$ be a height one prime ideal of $D$. Then $\bigcap_{\alpha \in X}\alpha \subseteq D_{\mathfrak{p}}$ if and only if $\alpha \subseteq D_{\mathfrak{p}}$ for some   $\alpha \in X$.
\end{lemma}

\begin{proof} 
Suppose $\bigcap_{\alpha \in X}\alpha \subseteq D_{\mathfrak{p}}$. 
We may assume $D \not \in X$.  By Lemma~\ref{directions}, there are finitely many rings $\alpha_1,\ldots,\alpha_n$
 in $Q_1(D)$ that are contained in $D_{\mathfrak{p}}$. 
 For each $i \in \{1,2,\ldots,n\}$, let $R_i$ be the intersection of the rings in $X$ that contain $\alpha_i$, where if there are no such rings, we let $R_i$ be the quotient field of $D$. Also, let $R$ be the intersection of the rings in $X$ that do not contain any of the $\alpha_i$, where $R$ is the quotient field of $D$ if there are no such rings. The intersection of the rings in $X$ is $R_1 \cap \cdots \cap R_n \cap R$.   
  By assumption, $R_1 \cap \cdots \cap R_n \cap R \subseteq D_{\mathfrak{p}}$, and so $R_1 D_{\mathfrak{p}} \cap  \cdots \cap R_nD_{\mathfrak{p}} \cap RD_{\mathfrak{p}} = D_{\mathfrak{p}}$. Since $D_{\mathfrak{p}}$ is a DVR it follows that $R \subseteq D_{\mathfrak{p}}$ or $R_i \subseteq D_{\mathfrak{p}}$ for some $i$. 
 
 We claim the   case    $R \subseteq D_{\mathfrak{p}}$  cannot occur. Indeed, suppose $R \subseteq D_{\mathfrak{p}}$.  Since the intersection of  the rings in $Q_1(D) \setminus \{\alpha_1,\ldots,\alpha_n\}$ is contained in $R$, this intersection is contained in  $D_{\mathfrak{p}}$. However, $Q_1(D) \setminus \{\alpha_1,\ldots,\alpha_n\}$ is a Noetherian subspace of the space of all local overrings with the Zariski  topology (this follows for example from \cite[Theorem~5.6]{MR4097907}), and hence the intersection of these rings commutes with localization \cite[Theorem~3.5]{MR2724221}. Therefore, if $R$ is contained in $D_{\mathfrak{p}}$, then as above one of the rings in $Q_1(D) \setminus \{\alpha_1,\ldots,\alpha_n\}$ is contained in $D_{\mathfrak{p}}$, a contradiction. This proves $R \not \subseteq D_{\mathfrak{p}}$ and so
  there is $i \in \{1,2,\ldots,n\}$ such that $R_i \subseteq D_{\mathfrak{p}}$.  
  
  Let $X(\alpha_i)$ be the set of rings in $X$ that contain $\alpha_i$. If $X(\alpha_i)$ consists of a single ring, then this ring is necessarily $R_i$ and so it is the ring in $X$ that is contained in $D_{\mathfrak{p}}$. Otherwise, since $R_i$ is the intersection of the rings in $X(\alpha_i)$ and $D_{\mathfrak{p}}$ is the localization of $\alpha_i$ at a height one prime,   we may repeat the argument with $D$ replaced by $\alpha_i$ and $X$ replaced by $X(\alpha_i)$. Since $X \subseteq Q_{\leq n}(D)$, after at most $n$ steps, each of which takes us up a level,
this process will produce a ring in $X$ contained in $D_{\mathfrak{p}}$. The converse of the lemma is clear. 
\end{proof}

\begin{theorem} \label{irredundant} 
Let $X $ be a nonempty subset of $Q(D)$. Every  ring in $Q(D)$ containing the intersection of the rings  in $X$ is comparable to a ring in $X$. 
\end{theorem} 

\begin{proof} Suppose $\alpha \in Q(D)$ and  $X$ is a nonempty subset of $Q(D)$ such that $\alpha$ is not comparable to any member in $X$.
  There is  $n >0$ such that $\alpha \in Q_n(D)$. 
    Consider the collection 
    $$Y= (X  \cap Q_{\leq n}(D)) \cup \{\beta \in Q_n(D):  \beta {\mbox{ is contained in a ring in }} X\}.$$ 
     Since $\alpha$ is not contained in any ring in $X$, $\alpha \not \in Y$.  
  Now $A(Y) \subseteq A(X)$, and so to prove $A(X) \not \subseteq \alpha$, it suffices to show $A(Y) \not \subseteq \alpha$. 
  
  Since $Q(D)$ is well-ordered, every ring in $Y$ contains a ring in $Y$ that is minimal in $Y$ with respect to inclusion. The set $\min Y$ of rings that are minimal in $Y$ consists of incomparable rings. Since no ring in $X$, hence no ring in $X \cap Q_{\leq n}(D)$, is contained in $\alpha$, it follows that no ring in $Y$ is contained in $\alpha$. (Note that the only rings in $Y$ that are not in $X \cap Q_{\leq n}(D)$ are in $Q_n(D)$ and hence are not comparable to $\alpha$.)
  Now $A(\min Y) = A(Y)$, so it suffices to show $A(\min Y) \not \subseteq \alpha$.  

Since $\{\alpha\} \cup \min Y$ consists of incomparable rings in $Q(D)$, 
Theorem~\ref{unique essential}
   implies 
  there is an essential valuation ring $V$ of $D$ that contains $\alpha$ but does not contain any ring in $\min Y$. 
  Thus  
 Lemma~\ref{localization} implies $V$ 
does not contain $A(Y) = A(\min Y)$, and so $A(Y) \not \subseteq \alpha$. 
\end{proof}

\begin{corollary} \label{main irred} If $X$ is a set of incomparable two-dimensional regular local overrings of $D$, then the intersection of the rings in $X$ is irredundant.   
\end{corollary} 

\begin{proof} This follows from Theorem~\ref{irredundant}. \end{proof}

Corollary~\ref{main irred} is proved in \cite[Theorem~8.6]{MR4097907} under the additional assumption that $D$ is Henselian. In addition to removing this restriction, the arguments in this section also fill a gap in the proof of Theorem~8.6 in \cite{MR4097907} involving a reduction made at the beginning of the proof. 

In higher-dimensional settings, although the analogous quadratic tree exists, Theorem~\ref{irredundant} fails.
For example, let $D$ be a $3$-dimensional RLR with $\mathfrak{m} = (x, y, z)$.
In the construction above, let $\mathfrak{q}$ be the prime ideal generated by the initial form of $x$.
Denote $D_y = D [\frac{x}{y}, \frac{z}{y}]$ and $D_z = D [\frac{x}{z}, \frac{y}{z}]$.
Then $X$ is the union of the set of localizations of $D_y$ at maximal ideals containing $(y, \frac{x}{y})$ and of localizations of $D_z$ at maximal ideals containing $(y, \frac{x}{z})$,
and the intersection of this ``projective line'' of points in $X$ is the ring $D$.

\section{Local intersections}

Let $x_1,\ldots,x_n \in D$. The
 {\it projective model of $D$} defined by $x_1,\ldots,x_n$  is the collection of local overrings of $D$ of the form 
 $$D\hspace{-.03in
 }\left[{x_1}/{x_i},\ldots,{x_n}/{x_i}\right]_P, {\mbox{ where }} 1 \leq i \leq n {\mbox{ and }}  P \in \Spec  D \hspace{-.03in
 }\left[{x_1}/{x_i},\ldots,{x_n}/{x_i}\right].$$
The Zariski topology of the projective model is defined analogously to that of $\Val(D)$ in the introduction, where now quantification is over the local rings in the model rather than  the valuation rings in $\Val(D)$.  
With this topology and a sheaf structure defined in terms of intersections of rings in open sets, the projective model of $D$ defined by $x_1,\ldots,x_n$ can be viewed as Proj$\:D[(x_1,\ldots,x_n)t],$ the blow up of $D$ along the ideal $(x_1,\ldots,x_n)$.  

Given two projective models $X$ and $Y$ of $D$, $Y$ {\it dominates} $X$ if every local ring in $X$  is dominated by a local ring in $Y$ and every local ring in $Y$ dominates a local ring in $X$. The {\it domination map} is the surjective mapping $d_{Y,X}:Y \rightarrow X$ that sends a local ring in $Y$ to the unique local ring in $X$ that it dominates. This map (which induces  a projective morphism of projective schemes over $D$) is continuous and closed in the Zariski  topology \cite[Lemma~4, p.~117]{MR0389876}. To simplify notation, if $Y'$ is a subset of $Y$, we write $X(Y')$ for the image  $d_{Y,X}( Y')$ of $Y'$ in $X$, and so $X(Y')$ is the collection of rings in $X$ dominated by the rings in $Y'$. We use this notation also when $Y'$ is a subset of the Zariski-Riemann space $\Zar(D)$ of $D$. In this case, we write simply $d_X$ for the domination map $d_{\Zar(D),X}:\Zar(D)\rightarrow X.$

If $X$ is a collection of local overrings rings of $D$, denote by  $A(X)$ the intersection of rings in $X$.
It follows from a more general result in   \cite[Theorem~4.1]{HLOT}
that if $X$ is a projective model of $D$ and $Y$ is a  closed connected subset of the set of local rings in $X$ that dominate $D$, then $A(Y)$ is a  local   ring that is dominated by each ring in $Y$. For this assertion, $D$ need not be assumed to be two-dimensional or even Noetherian. (All that is needed is that Spec$\:D$ is a Noetherian space.) 
We next prove a converse to this theorem, but for the converse the fact that $D$ is a two-dimensional regular local ring is crucial to our arguments. 

\begin{theorem} \label{Henselian lemma} Let  $X$ be a nonsingular projective model over $D$, and let ${X'}$ be a closed set of local rings in $X$, all of which  dominate $D$.  The intersection of the rings in $X'$ is a local ring if and only if ${X'}$ is connected in the Zariski topology.
\end{theorem} 

\begin{proof} 
Let $A = A({X'})$. If ${X'}$ is connected, then $A$ is local by the result cited before the theorem. 
To prove the converse, suppose $A$ is local. We show $X'$ is connected.   If $X'$ consists of only one ring, then clearly $X'$ is connected, so we assume for the rest of the proof that $X'$ contains more than one ring. 
We first claim $A$ is a  two-dimensional normal local ring that is essentially of finite type over $D$.  The closed points in $X'$ are the minimal elements of $X'$ with respect to set inclusion.  Thus $A$ is the intersection of the rings in $X'$ that are closed points in $X'$ and hence in $X$ since $X'$ is closed in $X$. Since $X$ is a nonsingular projective model of $F/D$, these closed points are two-dimensional regular local rings, and so $A$ is an intersection of two-dimensional regular local overrings in a projective model over $D$.  Therefore, 
by \cite[Theorem 7.2]{MR4097907}, $A$ is a two-dimensional normal Noetherian local ring that is essentially of finite type over $D$. 

Since $A$ is essentially of finite type over $D$, $A$ is a local ring in a normal projective model $Y$ over $D$. Let $W$ be the join of the projective models $X$ and $Y$, as defined in \cite[p.~120]{MR0389876}. (As a projective scheme, $W$ is the fiber product of $X$ and $Y$ over $\Spec(D)$.) Then $W$ dominates $X$ and $Y$, and we may consider the domination map $d_{W,Y}:W\rightarrow Y$ defined before the theorem. 
We will show 
\begin{enumerate} 
\item[(a)] $X'$ is a subset of $W$ that is closed in the Zariski topology, and 
\item[(b)] ${X'} = d_{W,Y}^{-1}(\{A\})$, i.e.~$X'$ is the set of local rings in $W$ that dominate $A$.
\end{enumerate}
Once the last assertion is proved, Zariski's Connectedness Theorem \cite[Corollary~11.3, p.~279]{MR0463157} and the fact that $A$ is integrally closed imply $X'$, as the fiber of $d_{W,Y}:W\rightarrow Y$ over the point $A \in Y$, is connected. 
Thus 
to prove the lemma, it remains to verify (a) and~(b). 

 To prove (a) it suffices to show $X'$ is a subset of $W$. This is because if
 $X'$ is a subset of $W$, 
then by properness (cf.~\cite[Lemma~4, p.~117]{MR0389876}), the only rings in $W$ dominating a ring in $X'$ are the rings in $X'$, so that 
 $X'= d_{W,X}^{-1}(X')$. Thus since $X'$ is closed in $X$ and $d_{W,X}$ is continuous, $X'$ is closed in $W$. Thus to verify (a), we show $X'$ is a subset of $W$.    
 
 Let $\alpha \in {X'}$.  
The local ring $\alpha$ lies in an affine submodel of $X$, and so there is a finitely generated  normal overring $R_1$  of $D$ such that $\alpha$ is a localization of $R_1$ and every localization of $R_1$ at a prime ideal is in $X$. (That we can assume $R_1$ is both normal and finitely generated follows from the fact that the normalization of a finitely generated overring of a regular local ring is finitely generated; see \cite{MR0126465}.)
Similarly, since  $A  \in Y$ there is a finitely generated normal overring $R_2$ of $D$ such that $A $ is a localization of $R_2$ at a prime ideal and every localization of $R_2$ lies in $Y$.  Each localization of the ring $R_1[R_2]$ at a prime ideal lies in the projective model $W$ \cite[Lemma 6, p.~120]{MR0389876}. 
Since  $A \subseteq \alpha$, we have $R_1 \subseteq R_1[R_2]  \subseteq R_1[A] \subseteq \alpha$. Thus, since $\alpha$ is a localization of $R_1$, 
$\alpha$ is a localization of $R_1[R_2]$, and so   $\alpha \in W$.  This shows ${X'} \subseteq W$, and hence (a) is verified. 

To prove (b), we must show   that $X'$
 is the set of local rings in $W$ dominating $A$. 
By the discussion preceding the statement of Theorem~\ref{Henselian lemma}, the local ring $A'$, as the  intersection of the rings in $X'$,  is dominated by all the rings in $X'$. 
Conversely, suppose 
  $\alpha \in W$ dominates $A$. We show $\alpha \in X'$.  
  Now $W$ dominates $Y$ and $A$ is a closed point in $Y$, so by Zariski's Connectedness Theorem, the set $W':=d^{-1}_{W,Y}(\{A\})$ of local rings  in $W$ dominating $A$ is connected. If the ring $\alpha$ is a closed point in $W'$ that is not contained in any other ring in $W'$, 
 then 
 $\alpha$ is an isolated point in $W'$, contrary to the fact that $W'$ is connected and contains more than one point (since $X' \subseteq W'$ and by assumption $X'$ contains more than one point).  
  Thus $\alpha$ is either a closed point in $W'$, hence a closed point in $W$, or $\alpha$ contains  a local ring in $W'$ that is a closed point in $W'$. 
As a step in proving $\alpha \in X'$, we prove first that 
  every two-dimensional local ring $\beta$ in $W$ that dominates $A $ and is contained in $\alpha$ is in $X'$. To this end, let $\beta $ be such a ring.

  Since every projective model over $D$ can be desingularized \cite{MR491722}, there is a two-dimensional regular local overring $\gamma$ of $D$  that dominates $\beta$ and hence dominates $A$. Since $X'$ is closed and, as discussed at the beginning of this section, the closed points of a projective model over $D$ all have dimension $2$,  $A=A(X')$ is the intersection of the two-dimensional local rings in $X'$. Each such ring is regular since $X$ is nonsingular. Thus $A$ is the intersection of the two-dimensional regular local rings in $X'$.  
 By Theorem~\ref{irredundant},
  each  local ring in $Q(D)$ that dominates $A$ is comparable to  a local ring in ${X'}$. 
 Thus $\gamma$ is comparable to a two-dimensional regular local ring $\delta$ in $X' \subseteq W$. 

 \smallskip

{\textsc{Case 1:}} $\delta \subseteq \gamma$. To verify this case, note that 
   as a two-dimensional local ring containing the two-dimensional local Noetherian domain $\delta$, $\gamma$ must dominate $\delta$. But $\gamma$ also dominates the ring $\beta \in W$. Let $V$ be a valuation ring of $F/D$ dominating $\gamma.$ Then $V$ dominates $\beta$ and $\delta$, and since  $\beta$ and $\delta$ are in the same projective model $W$, properness (cf.~\cite[Lemma~4, p.~117]{MR0389876}) implies  
    $\beta = d_W(V) = \delta$ and hence 
    $\beta \in {X'}$. 

\smallskip

{\textsc{Case 2:}} $\gamma \subseteq \delta$.
    In this case, $\delta$ dominates $\gamma$, which by assumption, dominates $\beta$. 
    Since $\delta \in X' \subseteq W$ and also  $\beta \in W' \subseteq W$, $\delta = \beta$ 
   since $\delta$ dominates $\beta$ and both rings are in the same projective model.  In this case also, $\beta \in X'$. 

   \smallskip

   Having verified that $\beta \in X'$ in all cases, 
we conclude that every two-dimensional  local ring   in $W$ that dominates $A $ and is contained in $\alpha$ is in $X'$.

   We now show  $\alpha$ is in $X'$. If $\dim(\alpha)=2$, then since $\beta \subseteq \alpha \in W$ and $\dim(\alpha)=\dim(\beta)$, we have  $\alpha = \beta$, so that $\alpha \in {X'}$. Otherwise, if $\dim(\alpha)=1$, then  every two-dimensional local ring in $W$ that is contained in $\alpha$ dominates $A$ since $\alpha$ dominates $A$.  We have shown that each such local ring is in ${X'}$. Since ${X'}$ is closed in $W$ and every local ring in $W$ contained in $\alpha$ is in ${X'}$, it follows that $\alpha$ is a generic point for a closed subset of $X'$, and hence $\alpha$ is in the closed set $X'$.  In all cases, $\alpha \in X'$, so  
this proves the claim that  ${X'}$ is the set of local rings in $W$ dominating $A$.  This verifies (b) and   completes the proof of the 
theorem.
\end{proof}

The next theorem differs from Theorem~\ref{VCT} in that the latter asserts that if $A$ is an integrally closed local overring of $D$, then 
{\it some}    connected subset $Z$ of $\Val(A)$ satisfies $A = A(Z)$. In Theorem~\ref{two dim}, it is shown that {\it every} subset $Z$ for which $A =A(Z)$ is connected. In this theorem and the next theorem, we use the  patch topology  discussed in (2.2).

\begin{theorem} \label{two dim}  \label{two dim cor}
If  ${Z}$ is a  subset of ${\rm Val}(D)$ such that $A({Z})$ is local, then ${Z}$ is   connected.
\end{theorem}

\begin{proof} 
Suppose  $A = A({Z})$ is local. If $A$ is not residually algebraic, then $A$ is a valuation domain by (2.1), and so $Z$ consists of overrings of $A$. Thus $Z$ is connected since these overrings of $V$ are totally ordered with respect to inclusion. 
Now suppose $A$ is residually algebraic over $D$. By \cite[Lemma~3.4]{HLOT}, to prove that $Z$ is connected, it suffices to show the Zariski closure $\overline{Z}$ of $Z$ is connected. Also, by \cite[Proposition~3.2]{HLOT}, the fact that $A$ is residually algebraic over $D$ implies $\overline{Z} \subseteq \Val(A)$, 
so 
$
A = A(Z) = A(\overline{Z})$. In this way, we can reduce to the case in which $Z$ is closed.

Suppose first that $A = A({Z})$ is local. By \cite[Lemma~3.6]{HLOT}, a nonempty patch closed  subset of $\Val(D)$ is connected if and only if its image in  each projective model  of $F/D$ is connected, so we will show the image of $Z$  in each projective model of $F/D$ is connected. 
Let $X$ be a projective model of $F/D$, and let $X^*$ be a desingularization of $X$, which exists by \cite{MR491722}.  Since $A$ is a local ring, so is $A(X^*({Z}))$ by \cite[Lemma~3.3(2)]{HLOT}. Also, since the domination map Zar$(D) \rightarrow X$ is closed, the image $X^*({Z})$ of ${Z}$ in $X^*$ is closed in $X^*$. Thus  $X^*({Z})$ is connected by Theorem~\ref{Henselian lemma}. Now $X(Z)$ is the image of $X^*(Z)$ under the domination map $X^* \rightarrow X$, and so, as the continuous image of a connected set, $X(Z)$ is connected, which proves $Z$ itself is connected 
  since 
the choice of projective model $X$ was arbitrary.  The converse follows from Theorem~\ref{is local}. 
\end{proof}

If $A$ is an integrally closed local overring of $D$, then $\Val(A)$ is a patch closed subset of $A$  and so $\Val(A)$ is a patch closed representation of $A$. 
The next theorem shows $\Val(A)$ is the only patch closed  representation of $A$, and hence also the only  Zariski closed representation of $A$. This is a strong  property that is special to our setting, as Example~\ref{3} illustrates.

\begin{theorem} \label{pre vacant} Let $A$ be an integrally closed local overring of $D$.   The only  patch closed  subset $Z$ of $\Val(A)$ with $A = A(Z)$ is   $\Val(A)$ itself.
\end{theorem}

\begin{proof} 
We can assume $A$ is not a  valuation ring since otherwise $\Val(A) = \{A\}$ and the assertion of the theorem clearly holds in this case. 

\medskip

\noindent
{\textsc{Claim 1}}: {\it There does not exist a proper Zariski closed subset $Z$ of $\Val(A)$ such that $A=A(Z)$. }

\medskip

Since $A$ is not a valuation ring, $A$ is residually algebraic over $D$ by {(2.1)}. Thus $\Val(A)$ is a Zariski closed subset of $\Val(D)$ by Theorem~\ref{is local}. 
Let $Z$ be a Zariski closed subset of $\Val(A)$ such that $A = A(Z)$. We will show $Z=\Val(A)$.
To this end, we first observe that every minimal valuation overring of $D$ in $\Val(A)$ is in $Z$. For if $V$ is a minimal valuation overring of $D$ in $\Val(A) \setminus Z$, then 
$Z':=Z \cup \{V\}$ is a disconnected closed subset of $\Val(D)$ for which $A = A(Z')$. However, Theorem~\ref{two dim} implies $Z'$ is connected since $A$ is local, a contradiction that implies $V \in Z$.
Thus every minimal valuation overring of $D$ in $\Val(A)$ is in $Z$.

To prove $Z = \Val(A)$, it remains to show that if $V \in \Val(A)$ is not a minimal valuation overring of $D$, then $V \in Z$. 
We have observed already that $\Val(A)$ is  closed in $\Val(D)$, so every minimal valuation overring of $D$ contained in $V$ is in $\Val(A)$. By (2.3), there are infinitely many minimal valuation overrings of $D$ contained in $V$, and by what we have shown, these minimal valuation overrings are in $Z$. By 
(2.2), 
 $V$ is in $Z$ since it is in the closure of the set of minimal valuation overrings of $D$ contained in $V$ and $Z$ is a Zariski closed set. 
With this, we conclude $Z = \Val(A)$. \qed

\medskip

\noindent {\textsc{Claim 2}}:
{\it The theorem holds if $A$ is Noetherian; i.e., if $A$ is Noetherian, the only patch closed subset $Z$ of $\Val(A)$ such that $A = A(Z)$ is $ \Val(A)$. }

\medskip

Under the assumption that $A$ is Noetherian, let $Z$ be a patch closed subset of $\Val(A)$ such that $A = A(Z)$.  
By (2.2), the Zariski closure $\overline{Z}$ of  $Z$ in $\Val(D)$ is \begin{center} $\overline{Z} = \{V \in \Val(D):V \subseteq U$ for some $U \in Z\}.$
\end{center} Since $A$ is not a valuation ring, $A$ is residually algebraic over $D$ by (2.1). By (2.2), $A = A(\overline{Z})$, and so 
 $\Val(A) = \overline{Z}$ by 
Claim 1.
 Thus 
  every  divisorial valuation ring of $D$ containing  $A$
  is in $Z$ since these rings are maximal in $\Val(A)$ with respect to set inlcusion. Since $A$ is Noetherian, the set of such valuation overrings of $A$ is patch dense in $\Val(A)$. (Since $\dim(A) \leq 2$, this follows, for example, from \cite[Lemma~3.3]{MR4186452}.)
  Thus  since $Z$ is patch closed in $\Val(A)$, we conclude  $Z = \Val(A)$. \qed

\medskip

\noindent {\textsc{Claim 3}}:
{\it Regardless of whether $A$ is Noetherian,  the only patch closed subset $Z$ of $\Val(A)$ such that $A = A(Z)$ is $\Val(A)$. }

\medskip

Assume $A$ is a not-necessarily-Noetherian integrally closed local overring of $D$.
Let $Z$ be a  patch closed  subset of $\Val(D)$ such that $A = A(Z)$. By Theorem~\ref{two dim}, $Z$ is Zariski connected. 
Let $X$ be a nonsingular projective model over $D$. 
Since $Z$ is connected and patch closed, so is its image  $X(Z)$ (see \cite[Remark 3.2]{MR3299704}). 
We will show $X(Z)$ is Zariski closed, but before proving this, we establish that $A(X(Z))$ is a  Noetherian local ring.

Since $A(Z)$ is a local ring, so is 
 $A(X(Z))$ \cite[Lemma~3.3(2)]{HLOT}.
 Since  $X(Z)$ consists of  local rings in a nonsingular projective model over $D$, the intersection of the closed points in $X(Z)$  is   Noetherian  by \cite[Theorem~7.4]{MR4097907}. Since $X$ is a projective model and the rings in $X(Z)$ dominate $D$, there are at most finitely many local rings  in $X(Z)$ that are not closed points. These rings are DVRs (they are in fact divisorial valuation rings of $D$), so $A(X(Z))$, as an intersection of an integrally closed Noetherian domain (namely, the intersection of the local rings that are the closed points in $X(Z)$) and finitely many DVRs, is a Krull domain. Every Krull overring of a two-dimensional Noetherian domain is Noetherian \cite[Theorem~9]{MR0254022}, so $A(X(Z))$ is a Noetherian domain.

 Returning to the proof that $X(Z)=d_X(Z)$ is   closed, we claim that
  $d_X^{-1}(X(Z))$ is a patch closed subset of $\Val(D)$. Since $d_X$ is continuous in the patch topology (see \cite[Remark~3.2]{MR3299704}), it suffices to verify that $X(Z)$ is patch closed in $X$. 
  And   this is indeed the case because $d_X$  is  continuous in the patch topology, $\Zar(D)$ is compact in the patch topology and  $X$ is Hausdorff in the patch topology, so that the image under $d_X$ of every patch closed subset of $\Zar(D)$ is patch closed in $X$. 

  We have already proved that $A(X(Z))$ is a 
 Noetherian local domain. By Claim~2,
using the fact that $d_X^{-1}(X(Z))$ is patch closed, we have  
 that $d_X^{-1}(X(Z))$ is the set $\Val(A(X(Z)))$, and so by Theorem~\ref{is local}, $d_X^{-1}(X(Z))$  is a Zariski closed subset of $\Val(D)$. Therefore, $X(Z)$, as the image of this set under the closed surjective map $d_X$, is also Zariski closed. 
Since $Z$ is patch closed, $Z = \bigcap_{X}d_X^{-1}(X(Z))$, where $X$ ranges over the nonsingular projective models  over $D$. (This follows from a standard topological argument; see the footnote in the proof of Lemma~3.5 in \cite{HLOT}.)   Since we have shown that for each nonsingular projective model $X$ of $D$, $X(Z)$ is Zariski closed, the fact that $d_X$ is continuous implies the set $d_X^{-1}(X(Z))$
is Zariski closed in $\Val(D)$.  Thus $Z$, as an intersection of Zariski closed sets, is Zariski closed. By Claim~1, the only Zariski closed subset of $\Val(A)$ whose members intersect to $A$ is $\Val(A)$, so $Z = \Val(A)$, which proves Claim~3 and hence completes the proof of the theorem.
\end{proof} 

\begin{example} \label{3} This example shows that the conclusion of Theorem~\ref{pre vacant} need not hold if we work over a ring of dimension higher than $2$. Let $k$ be a field, and let $x,y,z$
be indeterminates for $k$. Let 
\begin{itemize}
\item  $U = k(x,y)[z]_{(z)}$;

\item ${\mathfrak{M}}_U =$   maximal ideal of $U$;

\item $A = k + {\mathfrak{M}}_U$;

\item $\{V_\alpha\} =$  the  localizations of $k[x,y]$ at height one prime ideals;

\item $W_\alpha = V_\alpha + {\mathfrak{M}}_U$ for
each $\alpha$;

\item 
$V_1 = k[1/x,y]_{(1/x)}$ and $V_2 = k[x,1/y]_{(1/y)};$   and
\item
$W_1 = V_1 + {\mathfrak{M}}_U$ and $W_2 = V_2 + {\mathfrak{M}}_U$. 
\end{itemize}
Then $W_1,W_2$ and all the $W_\alpha$ are valuation rings contained in $U$ and have  quotient field $k(x,y,z)$. 
 Also, $A$ is an   integrally closed local domain with maximal ideal ${\mathfrak{M}}_U$.  The fact that the intersection of the $V_\alpha$ with $V_1$ and $V_2$ is $k$ implies  $A = A(Z)$, where $Z$ is union of the set $\{W_1,W_2,U\}$ with the set of the $W_\alpha$.   
Since the set that contains all the $V_\alpha$ and  $V_1$ and $V_2$ has the property that each nonzero element of $k(x,y)$ is a unit in all but finitely many of these valuation rings, it follows that $U$ is the unique patch limit point of the set $Z$. Thus $Z$ is a patch closed representation of $A$ that is properly contained in $\Val(A)$. 
\end{example}

The Zariski-Riemann space $\Zar(A)$ of a domain $A$ admits not only the Zariski and patch topologies  but also a topology that is dual to the Zariski topology, the {\it inverse topology} whose basic closed sets are  the Zariksi quasicompact  open sets of $\Zar(A)$. Inverse closed sets can be described similarly to how  Zariski closed sets are described in (2.1): A set $Z$ in $\Zar(A)$ is inverse closed if and only if there is a patch closed set $Y$ in $\Zar(A)$ such that
\begin{center} $Z = \{V \in \Zar(A):U \subseteq V$ for some $U$ in $Y\}$.
\end{center}
The domain $A$ is {\it vacant} if the only  inverse closed representation of $A$ in $\Zar(A)$ is $\Zar(A)$ itself. See \cite{MR2747239} for additional characterizations of vacant domains. Intuitively, a vacant domain is a  domain with ``few'' valuation overrings. By contrast, integrally 
Noetherian domains of dimension greater than one are from being vacant. 

\begin{corollary} \label{vacant cor} Every integrally closed local overring of $D$ of Krull dimension one is vacant. 
\end{corollary}

\begin{proof} Let $A$ be an integrally closed local overring of $D$ of Krull dimension one, and let $Z$ be an inverse closed subset of $\Zar(A)$ such that $A = A(Z)$. Since $A$ has dimension one, every valuation overring of $A$ except the quotient field $F$ of $A$ dominates $A$.  Thus 
$Z \cap \Val(A) = Z \setminus \{F\}$ is a patch closed subset of $\Val(A)$ such that $A$ is the intersection of the rings in $Z \cap \Val(A)$. By Theorem~\ref{pre vacant},  
the only patch closed representation of $A$ in
$\Val(A)$ is $\Val(A)$ itself, so 
$Z \setminus \{F\} = Z \cap \Val(A) = \Val(A)$. 
Thus $Z = \Val(A) \cup \{F\} = \Zar(A)$, which 
 proves
$A$ is vacant. 
\end{proof}

We can now prove  Theorem~\ref{nullstellensatz} from the introduction: 
{\it  The mappings 
  $Z \mapsto A(Z)$ and $A \mapsto \Val(A)$ 
define  a one-to-one correspondence between Zariski closed connected subsets of $\Val(D)$ and  integrally closed local overrings of $D$ that are not essential valuation rings of $D$ or divisorial valuation rings of $D$.}

\medskip

\noindent {\it Proof of Theorem~\ref{nullstellensatz}.}
If $Z$ is a closed connected subset of $\Val(D)$, then by Theorem~\ref{VCT}, $A(Z)$ is an integrally closed local ring that dominates $D$ and is residually algebraic over $D$, and hence is not a divisorial valuation ring of $D$ or an essential valuation ring. 
Conversely, suppose $A$ is an integrally closed local overring of $D$ that is not an essential valuation ring or divisorial valuation ring of $D$. Then $A$  dominates $D$ since otherwise 
 $A$ dominates a localization of $D$ 
at a height-one prime ideal and hence is
an essential valuation ring of $D$, contrary to assumption. Also, $A$ is  
 residually algebraic over $D$ by (2.1). 
 Therefore, by Theorem~\ref{VCT},  $\Val(A)$ is a closed connected subset of $\Val(D)$.

It remains to show the correspondence in Theorem~\ref{nullstellensatz} is one-to-one. Let $Z$ be a Zariski closed connected subset of $\Val(D)$. 
Then $Z$ is a patch closed subset of $\Val(D)$,  $A(Z)$ is local and $A(Z)$ is residually algebraic over $D$. 
We claim every valuation ring in $Z$ dominates $A(Z)$ and hence that $Z \subseteq \Val(A(Z))$. 

Let $U \in Z$. If $U$ is residually algebraic over $D$, then since $$(D+{\mathfrak M}_U)/{\mathfrak M}_U \subseteq (A(Z)+ {\mathfrak M}_U)/{\mathfrak M}_U \subseteq U/{\mathfrak M}_U$$ and $U/{\mathfrak M}_U$ is algebraic over the residue field of $D$, it follows that $(A(Z)+ {\mathfrak M}_U)/{\mathfrak M}_U$
is a field and hence $U$ dominates $A(Z)$. 
Suppose instead that $U$ is residually transcendental over $D$. Then the maximal ideal of $U$ is the intersection of the maximal ideals of the valuation rings in $\Val(D)$ properly contained in $U$, and since $Z$ is a closed subset of $\Val(D)$, all of these valuation rings properly contained in $U$ are in $Z$. This implies the intersection of all the maximal ideals 
 of the valuation rings in $Z$ 
 is the same as the intersection of the maximal ideals in the valuation rings in $Z$  that are minimal in $Z$. As noted above, since the latter valuation rings are residually algebraic over $D$, they all dominate $A(Z)$. This implies $U$ dominates $A(Z)$ since the maximal ideal of $U$ is the intersection of the maximal ideals of these valuation rings, and this proves the claim that
 every valuation ring in $Z$ dominates $A(Z)$.  

 We have shown that $Z \subseteq \Val(A(Z))$.  Since $Z$ is Zariski closed in $\Val(D)$, hence patch closed in $\Val(D)$, Theorem~\ref{pre vacant} implies 
 $Z = \Val(A(Z))$. Also, if $A$ is an integrally closed local overring of $D$,  then $A = A(\Val(A))$ since $A$ is the intersection of the valuation rings that dominate it, so the correspondence in Theorem~\ref{nullstellensatz} is one-to-one. 
\qed

\medskip

An integral domain $R$ is a {\it Pr\"ufer domain} if  each 
valuation overring of $R$ is a localization of $R$ at a prime ideal. 
Thus the Pr\"ufer overrings of the two-dimensional regular local ring $D$ can be viewed as the other extreme from the case of interest in this article, that of the infinite sets of  valuation rings that intersect to a local ring. We close this section with a characterization of such overrings. For this, we require the notion of a {\it Krull-constructed overring} of $D$, a local ring that is the intersection of the valuation rings contained in a divisorial valuation ring  in $\Val(D)$. These overrings of $D$ are precisely the rings that occur as $A(Z)$ for an infinite irreducible closed set $Z$ in $\Val(D)$.

\begin{theorem} An integrally closed overring of $D$ is a Pr\"ufer domain if and only if it is not contained in a Krull-constructed overring of $D$. 
\end{theorem}

\begin{proof} Let $A$ be an integrally closed overring of $D$. If $A$ is a Pr\"ufer domain, then $A$ is not contained in a Krull-constructed overring of $D$ since every overring of a Pr\"ufer domain is a Pr\"ufer domain and Krull-constructed overrings are not Pr\"ufer domains since they are local but none of their valuation overrings are localizations.

Conversely, suppose $A$ is not a Pr\"ufer domain, and let $M$ be a maximal ideal of $A$ such that $A_M$ is not a valuation domain. Seidenberg's Lemma \cite[Theorem~7]{MR54571} implies $A_M$ is dominated by a divisorial valuation ring $U$ of $D$. Since $A_M$ is not a valuation domain, (2.1) implies $A_M$ is residually algebraic over $D$, and so by Theorem~\ref{VCT}, $\Val(A_M)$ is a closed connected subset of $\Val(D)$. Thus the closure of the set $\{U\}$ in $\Val(D)$ is contained in $\Val(A_M)$, which implies 
 every rank two valuation ring contained in $U$ contains $A_M$. Consequently the Krull-constructed ring defined by $U$ contains $A_M$ and hence~$A$. 
 \end{proof}

\section{Questions}

The methods of this article are highly dependent on the fact that $D$ is a regular local ring of Krull dimension $2$.  
The following question is about whether the main theorem of the article, Theorem~\ref{nullstellensatz}, extends to higher dimensions. 

\begin{question} \label{7.1} Suppose $D$ is a regular local ring of Krull dimension greater than $2$.  Does there exist an integrally closed local overring $A$ of $D$ such that $A$ dominates $D$ and $A = A(Z)$ for some Zariski closed and disconnected subset $Z$ of $\Val(D)$? 
\end{question}

We do not know the answer to the question even if the assumption that $D$ is a regular local ring is replaced by the  assumption that $D$ is simply an integrally closed local domain of dimension more than $2$, Noetherian or otherwise.  One way to answer this question in the negative is to answer the following question in the positive, either by constructing the ring $A$ in Question~\ref{7.2} in general or, to answer Question~\ref{7.1} specifically, as an overring of a regular local ring of Krull dimension $>2$. 

\begin{question} \label{7.2}
Does there exist a local integral domain $A$ with quotient field $F$ and maximal ideal ${\mathfrak m}_A$ such that $A$ is not a valuation domain, $A= ({\mathfrak{m}}_A:_F{\mathfrak{m}}_A)$, and $A = R \cap V$, where $R$ is an intersection of valuation overrings dominating $A$ and $V$ is a valuation overring dominating $A$ such that $R \not \subseteq V$?
 \end{question}

A positive answer to Question~\ref{7.2} implies 
 there is a   closed and disconnected subset $Z$ of $\Val(A)$ such that $A = A(Z)$. To see this, 
note first that 
by replacing $V$ with a minimal valuation overring of $A$ contained in $V$, we may assume $\{V\}$ is a closed subset of $\Val(A)$.  
Let $Z$ be the set of valuation overrings of $R$ that are in $\Val(A)$.  
 By assumption, $A(Z)=R$. Let $\overline{Z}$ denote the Zariski closure of $Z$ in $\Val(A)$.   Since $A(\overline{Z}) \subseteq  R$, to prove the claim it suffices to show 
$A \ne A(\overline{Z})$, since this implies $V \not \in \overline{Z}$ and hence
$\overline{Z} \cup \{V\}$ is a   closed and disconnected set in $\Val(A)$ whose valuation rings intersect to $A$.    
By \cite[Lemma~3.1]{HLOT},  the intersection $J(Z)$ of the maximal ideals of the valuation rings in $Z$ is equal to  the intersection  $J(\overline{Z})$ of the maximal ideals of the valuation rings in $\overline{Z}$.  
Therefore, 
if $A({\overline{Z}}) = A$, then since each valuation ring in $\overline{Z}$ dominates $A$, it follows that $J(Z) = J( {\overline{Z}})$ is the maximal ideal ${\mathfrak{m}}_A$ of $A$.  
Thus   $R = A(Z) \subseteq ({\mathfrak m}_A:_F{\mathfrak m}_A) = A$, contrary to the assumption that $R \ne A$. This shows $A({\overline{Z}}) \ne A$, which proves the claim. 

\medskip

\printbibliography

\end{document}